\documentclass[reqno]{llncs}
\usepackage[utf8]{inputenc}
\usepackage{amsmath}
\usepackage{amsfonts}
\usepackage{amssymb}
\usepackage[svgnames]{xcolor}
\usepackage{colortbl}
\usepackage{iba-algo}

\makeatother

\numberwithin{equation}{section}

\makeatletter
\let\bfseries=\undefined
\DeclareRobustCommand\bfseries
{\not@math@alphabet\bfseries\mathbf
  \boldmath\fontseries\bfdefault\selectfont}
\makeatother

\newcommand{\Z}{{\mathbb Z}}

\newcommand{\R}{\mathbb R}

\DeclareMathOperator{\supp}{supp}

\DeclareMathOperator{\type}{type}

\def\Graver{{\mathcal G}}

\def\Circuits{{\mathcal C}}

\usepackage{enumerate}
\def\ve#1{\mathchoice{\mbox{\boldmath$\displaystyle\bf#1$}}
{\mbox{\boldmath$\textstyle\bf#1$}}
{\mbox{\boldmath$\scriptstyle\bf#1$}}
{\mbox{\boldmath$\scriptscriptstyle\bf#1$}}}

\newcommand\vecc{{\ve c}}

\newcommand\veg{{\ve g}}

\newcommand\vev{{\ve v}}

\newcommand\vex{{\ve x}}
\newcommand\vey{{\ve y}}

\newcommand{\maO}{O}            

\newcommand{\eoproof}{\hspace*{\fill} $\square$ \vspace{5pt}}

\usepackage{ifthen}
\newcommand{\DeclareBracket}[3]{
  \newcommand{#1}[2][]{%
  \ifthenelse%
  {\equal{##1}{}}%
  {\left#2##2\right#3}%
  {\csname ##1l\endcsname#2##2\csname ##1r\endcsname#3}}}
\DeclareBracket\set\{\}
\DeclareBracket\floor\lfloor\rfloor
\DeclareBracket\ceil\lceil\rceil
\DeclareBracket\bracket[]
\DeclareBracket\paren()

\newenvironment{psmallmatrix}{\left(\smallmatrix}{\endsmallmatrix\right)}

\newcommand\FourBlockBig[5][\relax]{\begin{pmatrix}#2& #3\\#4&#5 \end{pmatrix}\ifx#1\relax\else^{(#1)}\fi}
\newcommand\FourBlock[5][\relax]{\begin{psmallmatrix}#2& #3\\#4&#5 \end{psmallmatrix}\ifx#1\relax\else{^{(#1)}}\fi}

\newcommand\TwoBlock[3][\relax]{[#2,#3]\ifx#1\relax\else{^{(#1)}}\fi}

\newcommand{\T}{{\intercal}} 

\usepackage{hyperref}

\begin{document}
\pagestyle{headings}  

\title{Lower bounds on the Graver complexity of $M$-fold matrices}
\author{Elisabeth Finhold \and Raymond Hemmecke}
\institute{Technische Universit\"at M\"unchen, Germany}

\date{\today}

\maketitle

\begin{abstract}
  In this paper, we present a construction that turns certain relations on Graver basis elements of an $M$-fold matrix $A^{(M)}$ into relations on Graver basis elements of an $(M+1)$-fold matrix $A^{(M+1)}$. In doing so, we strengthen the bound on the Graver complexity of the $M$-fold matrix $A_{3\times M}$ from $g(A_{3\times M})\geq 17\cdot 2^{M-3}-7$ (Berstein and Onn) to $g(A_{3\times M})\geq 24\cdot 2^{M-3}-21$, for $M\geq 4$. Moreover, we give a lower bound on the Graver complexity $g(A^{(M)})$ of general $M$-fold matrices $A^{(M)}$ and we prove that the bound for $g(A_{3\times M})$ is not tight.
\end{abstract}

\vspace{-0.5cm}


\vspace{-0.5cm}

\section{Introduction}

The Graver basis of an integer matrix $A\in\Z^{d\times n}$ is a finite set of vectors $\Graver(A)\subseteq\ker(A)\cap\Z^n$ that allows the representation of any element in $\ker(A)\cap\Z^n$ as a nonnegative sign-compatible integer linear combination. This representation property is the basis for the application of Graver bases as optimality certificates for the minimization of linear and separable convex functions over the lattice points in a polyhedron \cite{DeLoera+Hemmecke+Koeppe:book,Graver:75,Murota+Saito+Weismantel}. In fact, if the polyhedron is defined by an $N$-fold matrix
\[
  [A,B]^{(N)}:=\begin{pmatrix}
    B & B & \cdots & B \\
    A & \maO  &   & \maO  \\
    \maO  & A &   & \maO  \\
       &   & \ddots &   \\
    \maO  & \maO  &   & A
  \end{pmatrix}
\]
composed of fixed matrices $A$ and $B$, then this minization problem is solvable in strongly-polynomial time \cite{DeLoera+Hemmecke+Onn+Weismantel,DeLoera+Hemmecke+Lee:SteepestDescent}. One fundamental result used in the proof is the fact that the Graver basis of $[A,B]^{(N)}$ does not get arbitrarily complicated: Any vector $\vex=(\vex^1,\dots,\vex^N)^\T\in\ker([A,B]^{(N)})$ consists of $N$ bricks $\vex^i\in\ker(A)$. By $\type(\vex)$ we denote the number $|\set{\,i:\vex^i\neq\ve 0\,}|$ of nonzero bricks in $\vex$. In \cite{Hosten+Sullivant,Santos+Sturmfels} it was shown that there is a constant $g(A,B)$, the so-called \emph{Graver complexity} of $A$ and $B$, given by
\begin{equation}\label{Eq: Formula for Graver complexity}
  \max\set{\,\|\veg\|_1:\veg\in\Graver(B\Graver(A))\,},
\end{equation}
such that the type of any Graver basis element in $\Graver([A,B]^{(N)})$ is bounded by $g(A,B)$ for any $N\in\Z_+$. This readily implies that for fixed matrices $A$ and $B$, the Graver basis $\Graver([A,B]^{(N)})$ has only a polynomial number $O(N^{g(A,B)})$ of elements and is computable in time $O(N^{g(A,B)})$ which is polynomial in $N$. Although an explicit construction of the Graver complexity of two matrices is stated in Eq.~(\ref{Eq: Formula for Graver complexity}), the iterated Graver basis computation $\Graver(B\Graver(A))$ renders this formula practically useless. So it is not very surprising that little is known about these constants $g(A,B)$, even in the special case that $B$ is the identity matrix of appropriate size, where we abbriviate $g(A):=g(A,I_n)$.

Taking as $A$ the node/edge-incidence matrix of a graph, then $g(A)$ gives a graph related constant. It appears (as $A_{L\times M}$) for complete bipartite graphs in the study of $3$-way tables of sizes $L\times M\times N$. Using Eq.~(\ref{Eq: Formula for Graver complexity}), one quickly computes $g(A_{3\times 3})=9$. However, already the next case of $g(A_{3\times 4})$ was out of reach for a long time. Only recently, it was shown that the conjectured value $g(A_{3\times 4})=27$ is indeed correct \cite{Finhold+Hemmecke+Kahle}. In \cite{Berstein+Onn}, the authors proved the first exponential lower bound $g(A_{3\times M})\geq 17\cdot 2^{M-3}-7$, for $M\geq 4$. This bound has been extended to general complete bipartite graphs \cite{Kudo+Takemura}.

The Graver complexities of matrices corresponding to certain monomial curves have been studied in \cite{Hemmecke+Nairn,Nairn}. In particular, it was proved that $g\left(\left(\begin{smallmatrix}1&1&1&1\\0&a&b&a+b\\\end{smallmatrix}\right)\right)=2(a+b)/\gcd(a,b)$ for $a,b\in\Z_{>0}$ and that the Graver and Gr\"obner complexities (= maximum type of any element in the universal Gr\"obner basis of the toric ideal associated to $[A,B]^{(N)}$) agree whenever $A$ is unimodular and $B$ is chosen arbitrarily. Finally, in \cite{Bogart+Hemmecke+Petrovic} it was shown that the Graver and Gr\"obner complexities agree although both bases do not coincide.

In this paper, we present a construction that turns certain relations on Graver basis elements of $A^{(M)}$ into relations on Graver basis elements of $A^{(M+1)}$. Applying this construction to the $M$-fold matrix $A_{3\times M}$, we can strengthen the bound $g(A_{3\times M})\geq 17\cdot 2^{M-3}-7$ (Berstein and Onn \cite{Berstein+Onn}) to $g(A_{3\times M})\geq 24\cdot 2^{M-3}-21$, for $M\geq 4$, and give also a lower bound on the Graver complexity $g(A^{(M)})$ of general $M$-fold matrices $A^{(M)}$. In fact, we show that for any fixed integer $M_0\geq 6$, we can strengthen our bound to $g(A_{3\times M})\geq (28-\frac{224}{2^{M_0}})\cdot 2^{M-3}-(2^{M_0-1}-7)$ for $M\geq M_0$.

The paper is structured as follows. First we state and prove our main result, Theorem \ref{Thm: New general bound}, in Section \ref{Section: Main result}, then we present some corollaries of it in Section \ref{Section: Corollaries} and finally, in Section 4, we apply our construction to $A_{3\times M}$ and show that our bound $g(A_{3\times M})\geq 24\cdot 2^{M-3}-21$ is still not tight.

\section{Main result}\label{Section: Main result}

A linear combination $\sum\limits_{i=1}^k h_i\cdot \vex^i=\ve 0$ on some integer vectors $\vex^i$ is \emph{primitive} if the coefficients $h_i$ are relatively prime nonzero integers and no $k-1$ of the vectors $\vex^1,\ldots,\vex^k$ satisfy any nontrivial linear relation. It is well-known that there can only be one primitive relation on each set of vectors:
\begin{lemma}\label{lemmaPrimRelCoefficients}
  Let $\sum\limits_{i=1}^k h_i \cdot\vex^i=\ve 0$ be a primitive relation, $\sum\limits_{i=1}^k s_i \cdot\vex^i=\ve 0$ another linear relation. Then \(s_i=\alpha \cdot h_i\) for all \(i=1,\ldots,k\) for some \(\alpha \in \R\).
\end{lemma}

An immediate consequence of Eq.~(\ref{Eq: Formula for Graver complexity}) is the following.
\begin{lemma}\label{lemmaPrimRelGraver}
  Let $C\in\Z^{d\times n}$ and suppose that $\sum\limits_{i=1}^k h_i \cdot \vex^i= 0$ is a primitive relation on $\vex^i \in \Graver(C)$. Then the Graver complexity of $C$ satisfies $g(C) \geq \sum\limits_{i=1}^k |h_i|$.
\end{lemma}

In the following, we consider matrices $C:=A^{(M)}\in \Z^{(c+M\cdot r) \times (M\cdot c)}$ 
for $A\in \Z^{r\times c}$ and view the corresponding vectors as $M\times c$ tables
\[
  \vex=\left(\begin{array}{c} \vex_1\\ \vdots \\ \vex_M\end{array} \right)\in \R^{M\times c}, \vex_i\in \R^{c}.
\]

With these notions, we are ready to state and prove our main result.

\begin{theorem}\label{Thm: New general bound}
  Let $\sum\limits_{i=0}^k h_i \cdot \vex^i=\ve 0$ be a primitive relation on $\vex^i\in\Graver\left(A^{(M)}\right)$ for some $M \in \Z_+$ with $h_i\in \Z_+$ that satisfies the following conditions:
  \begin{itemize}
	\item[(a)] the element $\vex^0$ is of the form
	$(\vex^0)^\T=(\vex_1,\ldots,\vex_{g-1},\ve 0,\ldots,\ve 0,\vex_g)^\T$,
	  with $\type(\vex^0)=:g>2$ and $\sum\limits_{i=1}^g 1\cdot\vex_i=\ve 0$ a primitive relation,
	\item[(b)] the last brick of $\vex^1,\ldots,\vex^l$ ($l\in \Z_+$) equals $-\vex_g$, and
  	\item[(c)] $h_0+s=h_1+\ldots +h_l$ for some $s\in \Z$ with $\gcd(g-1,s)=1$.
  \end{itemize}
  Then there is a primitive relation $\sum\limits_{i=0}^{k+g-1} \bar{h}_i\cdot \vey^i=0$ on $k+g$   elements  $\vey^i \in \Graver\left(A^{(M+1)}\right)$ such that
  \[
    \sum_{i=0}^{k+g-1}\bar{h}_i= 
    \begin{cases} 
    (g-1)\cdot \sum\limits_{i=0}^k h_i+ (2g-3)s & \text{if } s\geq 0,\\ 
    (g-1)\cdot \sum\limits_{i=0}^k h_i- s& \text{if } s<0,
    \end{cases}
  \]
  giving a lower bound on $g\left(A^{(M+1)}\right)$.
\end{theorem}

\begin{proof}
We first state the $k+g$ Graver basis elements $\vey^i\in\Graver\left(A^{(M+1)}\right)$. For this, let us denote by $(\vex^i)_j$ the $j$-th brick of $\vex^i$. We get the vectors $\vey^0,\ldots,\vey^l$ from $\vex^0,\ldots,\vex^l$ by moving the $M$-th row $(\vex^i)_M$ to the new brick $M+1$ and setting the $M$-th row to zero:
\[
  \vey^0:=\begin{pmatrix} (\vex^0)_1 \\ \vdots \\ (\vex^0)_{M-1} \\ \ve 0  \\ \hline(\vex^0)_M\end{pmatrix}
	=\begin{pmatrix} \vex_1 \\ \vex_2\\ \vdots \\ \vex_{g-1}\\ \ve 0\\ \vdots \\ \ve 0 \\\hline \vex_g\end{pmatrix}\qquad\text{and}\qquad  \vey^i:=\begin{pmatrix} (\vex^i)_1 \\ \vdots \\ (\vex^i)_{M-1} \\ \ve 0  \\ \hline(\vex^i)_M\end{pmatrix}
	=\begin{pmatrix} (\vex^i)_1 \\ \vdots \\ (\vex^i)_{M-1} \\  \ve 0  \\ \hline -\vex_g\end{pmatrix},
\]
for $i=1,\ldots,l$. 

For $i=l+1,\ldots,k$ the elements $\vey^{i}$ are defined by just appending a zero-brick to the vectors $\vex^i \in \Graver\left(A^{(M)}\right)$ to turn them into elements $\vey^i \in \Graver\left(A^{(M+1)}\right)$:
\[
  \vey^i:=\begin{pmatrix} (\vex^i)_1 \\ \vdots \\ (\vex^i)_M \\ \hline \ve 0\end{pmatrix}\qquad  i=l+1,\ldots,k.
\]
In order to define the $g-1$ new Graver basis elements $\vey^{k+1},\ldots, \vey^{k+g-1}$ we have to distinguish between the two cases $s\geq 0$ and $s<0$. For $s\geq 0$, we set
\[
  \vey^{k+1}:=\begin{pmatrix} \ve 0 \\ -\vex_2 \\ -\vex_3 \\ \vdots \\ -\vex_{g-1} \\ \ve 0 \\ \vdots \\ \ve 0 \\ -\vex_g \\ \hline -\vex_1\end{pmatrix}, 
  \vey^{k+2}:=\begin{pmatrix}-\vex_1 \\ \ve 0 \\ -\vex_3 \\ \vdots \\ -\vex_{g-1} \\ \ve 0 \\ \vdots \\ \ve 0 \\ -\vex_g \\ \hline -\vex_2\end{pmatrix}, \ldots,
  \vey^{k+g-1}:=\begin{pmatrix} - \vex_1 \\ -\vex_2 \\ \vdots \\ -\vex_{g-2}\\ \ve 0 \\ \ve 0 \\ \vdots \\ \ve 0 \\ -\vex_g \\ \hline -\vex_{g-1}\end{pmatrix}, 
\]
whereas for $s<0$ we change signs and set
\[
  \vey^{k+1}:=\begin{pmatrix} \ve 0 \\ \vex_2 \\  \vex_3 \\ \vdots \\ \vex_{g-1} \\ \ve 0 \\ \vdots \\ \ve 0 \\  \vex_g \\ \hline \vex_1\end{pmatrix}, 
  \vey^{k+2}:=\begin{pmatrix} \vex_1 \\ \ve 0 \\  \vex_3 \\ \vdots \\  \vex_{g-1} \\ \ve 0 \\ \vdots \\ \ve 0 \\  \vex_g \\ \hline \vex_2\end{pmatrix}, \ldots,
  \vey^{k+g-1}:=\begin{pmatrix} \vex_1 \\ \vex_2 \\ \vdots \\  \vex_{g-2}\\ \ve 0 \\ \ve 0 \\ \vdots \\ \ve 0 \\ \vex_g \\ \hline  \vex_{g-1}\end{pmatrix}.
\]

We claim that
\[
  \sum_{i=0}^{k+g-1}\bar{h}_i\cdot\vey^i:= [(g-1) h_0+(g-2) s]\cdot\vey^0+\sum_{i=1}^{k}(g-1) h_i \cdot \vey^i+\sum_{i=k+1}^{k+g-1}  |s| \cdot\vey^i=\ve 0
\]
is a primitive relation on the elements $\vey^0,\ldots,\vey^{k+g-1}$. (It can easily be verified that these elements indeed sum up to zero given the specified coefficients/weights.)

Let $\sum\limits_{i=0}^{k+g-1} \alpha_i \cdot\vey^i=\ve 0$ be any relation on the $\vey^i$. We show that the coefficients $\alpha_i$ are uniquely determined up to some factor $\alpha$. This implies that the relation above is indeed primitive. Looking at the $(M+1)$-th bricks, we conclude that
\begin{align*}
	  \ve 0= \sum\limits_{i=0}^{k+g-1} \alpha_i \cdot(\vey^i)_{M+1}
		=\begin{cases}
						\alpha_0\cdot\vex_g+\sum\limits_{i=1}^{l} \alpha_i \cdot (-\vex_g)+\sum\limits_{i=k+1}^{k+g-1} \alpha_i\cdot (- \vex_{i-k})\\
	\qquad =(\alpha_0-\sum\limits_{i=1}^l \alpha_i)\cdot\vex_g+\sum\limits_{i=1}^{g-1} (-\alpha_{k+i})\cdot\vex_i & \text{if } s\geq 0,\\
				\alpha_0\cdot\vex_g+\sum\limits_{i=1}^{l} \alpha_i \cdot (-\vex_g)+\sum\limits_{i=k+1}^{k+g-1} \alpha_i\cdot  \vex_{i-k}\\
	\qquad =(\alpha_0-\sum\limits_{i=1}^l \alpha_i)\cdot\vex_g+\sum\limits_{i=1}^{g-1} \alpha_{k+i}\cdot\vex_i & \text{if } s< 0.
		\end{cases}
\end{align*}
As the relation $\sum\limits_{i=1}^{g} 1\cdot\vex_i$ is primitive, Lemma  \ref{lemmaPrimRelCoefficients} implies
$ -\alpha_{k+1}=\cdots=-\alpha_{k+g-1}=\alpha_0-\sum\limits_{i=1}^l \alpha_i$
if $s \geq 0$,  and  $\alpha_{k+1}=\cdots=\alpha_{k+g-1}=\alpha_0-\sum\limits_{i=1}^l \alpha_i$ if $s<0$. With 
\begin{align}
	t:=\alpha_{k+1}=\ldots=\alpha_{k+g-1}
	=\begin{cases} -\alpha_0+ \alpha_1+\ldots+\alpha_l& \text{if } s\geq 0, \\
	\alpha_0- \alpha_1-\ldots-\alpha_l& \text{if } s< 0,\end{cases}\label{eq3}
\end{align}
we get
\[
	 \ve 0= \sum\limits_{i=0}^{k+g-1} \alpha_i \cdot\vey^i=\sum\limits_{i=0}^{k} \alpha_i \cdot\vey^i+t\cdot \sum_{i=k+1}^{k+g-1} \vey^i.
\]
We can turn this relation on elements of type $M+1$ into a relation on elements of type $M$ by replacing the last two bricks of each vector by their sum (= a single brick). Observe that by doing this, the vectors $\vey^i$, $i=0,\ldots,k$, just turn into the corresponding $\vex^i$, whereas the sum $\vey^{k+1}+\ldots+\vey^{k+g-1}$ by definition of the new vectors $\vey^{k+1},\ldots,\vey^{k+g-1}$ becomes $-(g-2)\cdot \vex^0$ if $s\geq 0$, and $(g-2)\cdot \vex^0$ in case $s< 0$. That is why $\sum\limits_{i=0}^{k+g-1} \alpha_i \cdot \vey^i=\ve 0$ implies that
\[
	\ve 0 = \begin{cases}
	\sum\limits_{i=0}^{k} \alpha_i \cdot\vex^i- t (g-2)\cdot\vex^0 =[\alpha_0- (g-2)t] \cdot \vex^0+\sum\limits_{i=1}^k \alpha_i \cdot\vex^i & \text{if } s\geq 0,\\
	\sum\limits_{i=0}^{k} \alpha_i \cdot\vex^i+ t (g-2)\cdot\vex^0 =[\alpha_0+ (g-2)t] \cdot \vex^0+\sum\limits_{i=1}^k \alpha_i \cdot \vex^i & \text{if } s<0.	
	\end{cases}
\]
But $\sum\limits_{i=0}^{k} h_i\cdot \vex^i=0$ is primitive by assumption, and so again by Lemma \ref{lemmaPrimRelCoefficients} it must hold that 
\begin{align}
	\alpha\cdot h_0&=\begin{cases}
		\alpha_0-(g-2) t & \text{if } s\geq 0,\\
		\alpha_0+(g-2) t& \text{if } s<0,	
	\end{cases}	\label{eq1} \\
	 \alpha \cdot h_i &=\alpha_i,\hfill &i=1,\ldots,k  	\label{eq4} 
\end{align}
for some $\alpha$.
Therefore, by using assumption (c), we have that 
\begin{equation}
	\alpha\cdot(h_0+s)\stackrel{\text{(c)}}{=}\alpha\cdot (h_1+\ldots+h_l)	\stackrel{(\ref{eq4})}{=}\alpha_1+\ldots+\alpha_l\stackrel{(\ref{eq3})}{=}
	\begin{cases} \alpha_0+ t & \text{if } s\geq 0,\\ \alpha_0-t & \text{if } s< 0.\end{cases} \label{eq2}
\end{equation}
Subtracting equations (\ref{eq1}) and (\ref{eq2}) gives 
\[
	(g-1) t = \begin{cases}
	 \alpha\cdot(h_0+s)- \alpha \cdot h_0 & \text{if } s\geq 0,\\
	- \alpha\cdot(h_0+s)+ \alpha \cdot h_0 & \text{if } s< 0,
	\end{cases}
\]
and therefore
\[
	t	= \alpha\cdot\frac{|s|}{g-1}.
\] 
Plugging in this value for $t$ in (\ref{eq1}), (\ref{eq4}) and (\ref{eq3}), we get that in any relation $\sum\limits_{i=0}^{k+g-1} \alpha_i\cdot \vey^i=\ve 0$ the coefficients have to satisfy
\begin{align*}
	\alpha_0&=\alpha\cdot\left(h_0+ \frac{g-2}{g-1}s\right),\\
	\alpha_i&=\alpha\cdot h_i  \hfill &i=1,\ldots,k,\\
	\alpha_i&= \alpha \cdot\frac{|s|}{g-1} \hfill &i=k+1,\ldots,k+g-1.
\end{align*}
Thus any nontrivial relation must involve all $\vey^i$ and therefore our relation $\sum\limits_{i=0}^{k+g-1} \bar{h}_i\cdot \vey^i=\ve 0$ is support minimal. 

Such a support minimal relation is primitive if and only if all coefficients are nonzero integers and do not have a common factor, which is the case for
\begin{align*}
	\bar{h}_0&=(g-1)h_0+ (g-2)s,\\
	\bar{h}_i&=(g-1) h_i, \hfill &i=& 1,\ldots,k,\\
	\bar{h}_i&= |s|,  \hfill &i=& k+1,\ldots,k+g-1\;,
\end{align*}
since any common factor $u$ of $\bar{h}_0,\ldots,\bar{h}_{k+g-1}$ clearly divides $s$. From $\gcd(g-1,s)=1$, we conclude that $u$ is a divisor of $h_1,\ldots,h_k$ and also of $h_0$. As $h_0,\ldots,h_k$ are co-prime (they are used in the primitive relation $\sum\limits_{i=0}^{k} {h}_i\cdot \vex^i=\ve 0$ by assumption), we must have $u=1$. Finally we can compute  
\begin{align*}
  \sum_{i=0}^{k+g-1}\bar{h}_i&= ((g-1) h_0+(g-2) s) +\sum_{i=1}^k(g-1)h_i + \sum_{i=k+1}^{k+g-1} |s| \\ &=
{\small
	\begin{cases} 
		(g-1)\sum\limits_{i=0}^k h_i + (g-2)s+(g-1)s=(g-1)\sum\limits_{i=0}^k h_i +(2g-3)s	& \text{if } s\geq 0,\\
		(g-1)\sum\limits_{i=0}^k h_i + (g-2)s - (g-1)s= (g-1)\sum\limits_{i=0}^k h_i -s 			& \text{if } s<0. 	
	\end{cases}
}
\end{align*}
\eoproof
\end{proof}

\section{Corollaries}\label{Section: Corollaries}

The following corollary shows that Theorem \ref{Thm: New general bound} can be applied recursively, which immediately leads to a lower bound for $g\left(A^{(M)}\right)$.
\begin{corollary}\label{Cor: Recursion}
Let $\sum\limits_{i=0}^k h_i \cdot \vex^i=\ve 0$ be a primitive relation on $\vex^i\in\Graver\left(A^{(M_0)}\right)$ for some $M_0\in\Z_+$, with  $s$, $g=\type(\vex^0)$, satisfying the conditions of Theorem \ref{Thm: New general bound}. 

Then for $M\in \Z, M\geq M_0$, the Graver complexity $g\left(A^{(M)}\right)$ of the $M$-fold matrix $A^{(M)}$ is bounded from below by
\[
	g\left(A^{(M)}\right)\geq 
	\begin{cases}
		(g-1)^{M-M_0}\cdot\left(\sum\limits_{i=0}^k h_i +\frac{(2g-3)s}{g-2} \right) -\frac{(2g-3)s}{g-2} & \text{if } s\geq 0, \\
		(g-1)^{M-M_0}\cdot\left(\sum\limits_{i=0}^k h_i +\frac{-s}{g-2} \right) +\frac{s}{g-2} & \text{if } s < 0. \\
	\end{cases}
\]
\end{corollary}

\begin{proof}
Observe that the new relation $\sum\limits_{i=0}^{k+g-1}\bar{h}_i\cdot\vey^i=\ve 0$ defined in the proof of Theorem \ref{Thm: New general bound} again satisfies the conditions of Theorem \ref{Thm: New general bound}, i.e. 
\begin{itemize}
	\item[(a)] $(\vey^0)^\T=(\vex_1,\ldots,\vex_{g-1},\ve 0,\ldots,\ve 0,\ve 0,\vex_g)^\T$,
	\item[(b)] row $M+1$ of $\vey^1,\ldots,\vey^l$ equals $-\vex_g$, and 	 
	\item[(c)] $\bar{h}_0+s=(g-1)h_0+(g-2)s+s= (g-1)(h_0+s)=(g-1)(h_1+\ldots,h_l)=\bar{h}_1+\ldots+\bar{h}_l$.
\end{itemize}
Hence, we can apply Theorem \ref{Thm: New general bound} recursively and by induction we obtain the formula stated as a lower bound of $A^{(M)}$. 
\eoproof
\end{proof}

Applying this corollary to a more concrete base case $\sum\limits_{i=0}^{k}{h}_i\cdot\vex^i=\ve 0$, we obtain the following bound.
\begin{corollary}\label{Cor: General bound}
Let $g:= \max\set{\,|\supp(\vecc)|:\vecc\in\Circuits(\Graver(A))\,}\geq 3$. Then for $M\in \Z$,  $M\geq g$, the Graver complexity $g\left(A^{(M)}\right)$ of the $M$-fold matrix $A^{(M)}$ is bounded from below by
\[
	g\left(A^{(M)}\right)\geq \frac{g-1}{g-2}\cdot (g-1)^{M-(g-1)}-\frac{1}{g-2}.
\]
\end{corollary}

\begin{proof} By definition of $g$ there is a primitive relation $\sum\limits_{i=1}^gc_i\cdot \vev_i$ with $\vev_i\in\Graver(A)$. With this, we define $\vex^0:=\left(\begin{array}{c}\vex_1 \\ \vdots \\ \vex_g\end{array}\right):=\left(\begin{array}{c}c_1\vev_1 \\ \vdots \\ c_g\vev_g\end{array}\right)\in \Graver\left(A^{(g)} \right)$ with $\type(\vex^0)=g$ and $\sum\limits_{i=1}^{g} 1\cdot \vex_i$ being a primitive relation. Thus 
\[
	1\cdot \left(\begin{array}{c}\vex_1 \\ \vex_2\\ \vex_3\\  \vdots \\ \vex_{g-1}\\ \vex_g\end{array}\right) 
	+1\cdot\left(\begin{array}{c}\vex_2 \\ \vex_3\\ \vex_4\\  \vdots \\ \vex_g \\ \vex_1 \end{array}\right)
	+1\cdot \left(\begin{array}{c}\vex_3 \\ \vex_4\\ \vex_5\\  \vdots \\ \vex_1 \\ \vex_2 \end{array}\right)
\ldots
	+1\cdot \left(\begin{array}{c}\vex_g \\ \vex_1\\ \vex_2\\  \vdots \\ \vex_{g-2} \\ \vex_{g-1}\end{array}\right)=0
\]
is a primitive relation on $g$ elements $\vex^i\in \Graver\left(A^{(g)}\right), i=0,\ldots,g-1$. We choose 
$M_0=g$, $l=0$ and $s=-h_0=-1$. 
Then  we apply Corollary \ref{Cor: Recursion} and get
\[
	g\left(A^{(M)}\right)\geq  (g-1)^{M-g}\left(\sum\limits_{i=0}^{g-1} 1 +\frac{1}{g-2}\right)-\frac{1}{g-2}=\frac{g-1}{g-2}\cdot (g-1)^{M-(g-1)}-\frac{1}{g-2}.
\]
\eoproof
\end{proof}

Note that for $A_{3\times M}$ we have $g=3$ which gives a bound $g\left(A_{3\times M}\right)\geq  4\cdot 2^{M-3}-1$, which already shows the known exponential behavior of the lower bound. However, choosing a better primitive relation, we can get a much better bound.

\begin{corollary}\label{Cor: A3M bound} 
\[
	g\left(A_{3\times M}\right)\geq  24\cdot 2^{M-3}-21, \qquad\text{ for } M\geq 4.
\]
\end{corollary}
\begin{proof}Let us take the following primitive relation $\sum\limits_{i=0}^6h_i\vex^i=\ve 0$ on $7$ elements $\vex^0,\ldots,\vex^6\in \Graver(A_{3\times 4})$ (see \cite{Hemmecke+Nairn}): 
{\footnotesize
\begin{align*}
 &1\cdot \left(\begin{array}{rrr} 0 & -1 & 1 \\ 1 & 0 & -1 \\ 0 & 0 & 0 \\  -1 & 1 & 0 \end{array} \right)
+ 3\cdot \left(\begin{array}{rrr} 0 & 0 & 0 \\ 0 & 1 & -1\\-1 & 0 & 1 \\ 1 & -1 & 0\end{array} \right)
+ 5\cdot \left(\begin{array}{rrr} -1 & 0 & 1 \\ 0 & 0 & 0 \\ 0 & 1 & -1 \\ 1 & -1 & 0 \end{array}\right)\\
+ &2\cdot \left(\begin{array}{rrr} -1 & 1 & 0 \\ 0 & 0 & 0 \\ 0 & -1 & 1 \\ 1 & 0 & -1\end{array} \right)
+ 3\cdot \left(\begin{array}{rrr} 0 & 0 & 0 \\ 0 & 1 & -1 \\ 1 & -1 & 0 \\ -1 & 0 & 1\end{array} \right)
+ 6\cdot \left(\begin{array}{rrr} 0 & 1 & -1 \\ 1 & -1 & 0 \\ 0 & 0 & 0 \\ -1 & 0 & 1 \end{array} \right)
+ 7\cdot \left(\begin{array}{rrr} 1 & -1 & 0 \\ -1 & 0 & 1 \\ 0 & 0 & 0 \\ 0 & 1 & -1 \end{array} \right)= 0.
\end{align*}
}
Applying Corollary \ref{Cor: Recursion} with $l=2$, $s=7$, $g=3$ gives $g\left(A_{3\times M}\right)\geq 2^{M-4}\cdot\left(27+\frac{21}{1}\right) -\frac{21}{1}= 24\cdot 2^{M-3}-21$.
\eoproof
\end{proof}

\section{Sample constructions}\label{Section: Examples}

Applying the construction in Theorem \ref{Thm: New general bound} recursively, we obtain the following primitive relations among elements in $\Graver(A_{3\times M})$, $M=5,6,7$.

\smallskip

$M=5$
{\footnotesize
\begin{align*}
 9\cdot &\left(\begin{array}{rrr} 0 & -1 & 1 \\ 1 & 0 & -1 \\ 0 & 0 & 0 \\ \textbf{0} & \textbf{0} & \textbf{0}\\  -1 & 1 & 0 \end{array} \right)
+ 6\cdot \left(\begin{array}{rrr} 0 & 0 & 0 \\ 0 & 1 & -1\\-1 & 0 & 1 \\\textbf{0} & \textbf{0} & \textbf{0}\\ 1 & -1 & 0\end{array} \right)
+ 10\cdot \left(\begin{array}{rrr} -1 & 0 & 1 \\ 0 & 0 & 0 \\ 0 & 1 & -1 \\ \textbf{0} & \textbf{0} & \textbf{0}\\ 1 & -1 & 0 \end{array} \right)\\
+ 4\cdot &\left(\begin{array}{rrr} -1 & 1 & 0 \\ 0 & 0 & 0 \\ 0 & -1 & 1 \\ 1 & 0 & -1\\\textbf{0} & \textbf{0} & \textbf{0}\end{array} \right)
+ 6\cdot \left(\begin{array}{rrr} 0 & 0 & 0 \\ 0 & 1 & -1 \\ 1 & -1 & 0 \\ -1 & 0 & 1\\ \textbf{0} & \textbf{0} & \textbf{0}\end{array} \right)
+ 12\cdot \left(\begin{array}{rrr} 0 & 1 & -1 \\ 1 & -1 & 0 \\ 0 & 0 & 0 \\ -1 & 0 & 1\\ \textbf{0} & \textbf{0} & \textbf{0} \end{array} \right)
+ 14\cdot \left(\begin{array}{rrr} 1 & -1 & 0 \\ -1 & 0 & 1 \\ 0 & 0 & 0 \\ 0 & 1 & -1\\ \textbf{0} & \textbf{0} & \textbf{0} \end{array} \right) \\
+ 7\cdot &\left(\begin{array}{rrr}  0 & 0 & 0 \\ -1 & 0 & 1 \\  0 & 0 & 0\\ 1 & -1 & 0 \\ 0 & 1 & -1 \end{array} \right)
+ 7\cdot \left(\begin{array}{rrr}  0 & 1 & -1 \\ 0 & 0 & 0  \\  0 & 0 & 0\\ 1 & -1 & 0 \\  -1 & 0 & 1\end{array} \right)
=0
\end{align*}
}
with $\sum_i\bar h_i=9+6+10+4+6+12+14+7+7=75$.

\smallskip

$M=6$
{\footnotesize
\begin{align*}
 25\cdot &\left(\begin{array}{rrr} 0 & -1 & 1 \\ 1 & 0 & -1 \\ 0 & 0 & 0 \\ 0 & 0 & 0\\\textbf{0} & \textbf{0} & \textbf{0}\\  -1 & 1 & 0 \end{array} \right)
+ 12\cdot \left(\begin{array}{rrr} 0 & 0 & 0 \\ 0 & 1 & -1\\-1 & 0 & 1 \\ 0 & 0 & 0\\ \textbf{0} & \textbf{0} & \textbf{0}\\ 1 & -1 & 0\end{array} \right)
+ 20\cdot \left(\begin{array}{rrr} -1 & 0 & 1 \\ 0 & 0 & 0 \\ 0 & 1 & -1 \\ 0 & 0 & 0\\ \textbf{0} & \textbf{0} & \textbf{0}\\ 1 & -1 & 0 \end{array} \right)\\
+ 8\cdot &\left(\begin{array}{rrr} -1 & 1 & 0 \\ 0 & 0 & 0 \\ 0 & -1 & 1 \\ 1 & 0 & -1\\ 0 & 0 & 0 \\\textbf{0} & \textbf{0} & \textbf{0}\end{array} \right)
+ 12\cdot \left(\begin{array}{rrr} 0 & 0 & 0 \\ 0 & 1 & -1 \\ 1 & -1 & 0 \\ -1 & 0 & 1\\ 0 & 0 & 0 \\ \textbf{0} & \textbf{0} & \textbf{0}\end{array} \right)
+ 24\cdot \left(\begin{array}{rrr} 0 & 1 & -1 \\ 1 & -1 & 0 \\ 0 & 0 & 0 \\ -1 & 0 & 1\\ 0 & 0 & 0 \\ \textbf{0} & \textbf{0} & \textbf{0} \end{array} \right)
+ 28\cdot \left(\begin{array}{rrr} 1 & -1 & 0 \\ -1 & 0 & 1 \\ 0 & 0 & 0 \\ 0 & 1 & -1\\ 0 & 0 & 0 \\ \textbf{0} & \textbf{0} & \textbf{0}\end{array} \right) \\
+ 14\cdot &\left(\begin{array}{rrr}  0 & 0 & 0 \\ -1 & 0 & 1 \\  0 & 0 & 0\\ 1 & -1 & 0 \\ 0 & 1 & -1\\ \textbf{0} & \textbf{0} & \textbf{0} \end{array} \right)
+ 14\cdot \left(\begin{array}{rrr}  0 & 1 & -1 \\ 0 & 0 & 0  \\  0 & 0 & 0\\ 1 & -1 & 0 \\  -1 & 0 & 1\\ \textbf{0} & \textbf{0} & \textbf{0}\end{array} \right) 
+  7\cdot \left(\begin{array}{rrr}  0 & 0 & 0 \\ -1 & 0 & 1 \\  0 & 0 & 0\\ 0 & 0 & 0\\ 1 & -1 & 0 \\ 0 & 1 & -1 \end{array} \right)
+ 7\cdot \left(\begin{array}{rrr}  0 & 1 & -1 \\ 0 & 0 & 0  \\  0 & 0 & 0\\ 0 & 0 & 0\\ 1 & -1 & 0 \\  -1 & 0 & 1\end{array} \right) 
=0
\end{align*}
}
with $\sum_i\bar h_i=25+12+20+8+12+24+28+14+14+7+7=171$.

\smallskip

$M=7$
{\footnotesize
\begin{align*}
 57\cdot &\left(\begin{array}{rrr} 0 & -1 & 1 \\ 1 & 0 & -1 \\ 0 & 0 & 0 \\ 0 & 0 & 0\\0 & 0 & 0\\ \textbf{0} & \textbf{0} & \textbf{0}\\  -1 & 1 & 0 \end{array} \right)
+ 24\cdot \left(\begin{array}{rrr} 0 & 0 & 0 \\ 0 & 1 & -1\\-1 & 0 & 1 \\ 0 & 0 & 0\\0 & 0 & 0\\ \textbf{0} & \textbf{0} & \textbf{0}\\ 1 & -1 & 0\end{array} \right)
+ 40\cdot \left(\begin{array}{rrr} -1 & 0 & 1 \\ 0 & 0 & 0 \\ 0 & 1 & -1 \\ 0 & 0 & 0\\0 & 0 & 0\\ \textbf{0} & \textbf{0} & \textbf{0}\\ 1 & -1 & 0 \end{array} \right)\\
+ 16\cdot &\left(\begin{array}{rrr} -1 & 1 & 0 \\ 0 & 0 & 0 \\ 0 & -1 & 1 \\ 1 & 0 & -1\\ 0 & 0 & 0 \\ 0 & 0 & 0\\\textbf{0} & \textbf{0} & \textbf{0}\end{array} \right)
+ 24\cdot \left(\begin{array}{rrr} 0 & 0 & 0 \\ 0 & 1 & -1 \\ 1 & -1 & 0 \\ -1 & 0 & 1\\ 0 & 0 & 0 \\ 0 & 0 & 0\\ \textbf{0} & \textbf{0} & \textbf{0}\end{array} \right)
+ 48\cdot \left(\begin{array}{rrr} 0 & 1 & -1 \\ 1 & -1 & 0 \\ 0 & 0 & 0 \\ -1 & 0 & 1\\ 0 & 0 & 0 \\ 0 & 0 & 0 \\ \textbf{0} & \textbf{0} & \textbf{0}\end{array} \right)
+ 56\cdot \left(\begin{array}{rrr} 1 & -1 & 0 \\ -1 & 0 & 1 \\ 0 & 0 & 0 \\ 0 & 1 & -1\\ 0 & 0 & 0 \\ 0 & 0 & 0\\\textbf{0} & \textbf{0} & \textbf{0}\end{array} \right) \\
+ 28\cdot &\left(\begin{array}{rrr}  0 & 0 & 0 \\ -1 & 0 & 1 \\  0 & 0 & 0\\ 1 & -1 & 0 \\ 0 & 1 & -1\\ 0 & 0 & 0 \\ \textbf{0} & \textbf{0} & \textbf{0}\end{array} \right)
+ 28\cdot \left(\begin{array}{rrr}  0 & 1 & -1 \\ 0 & 0 & 0  \\  0 & 0 & 0\\ 1 & -1 & 0 \\  -1 & 0 & 1\\ 0 & 0 & 0\\ \textbf{0} & \textbf{0} & \textbf{0}\end{array} \right) 
+ 14\cdot \left(\begin{array}{rrr}  0 & 0 & 0 \\ -1 & 0 & 1 \\  0 & 0 & 0\\ 0 & 0 & 0\\ 1 & -1 & 0 \\ 0 & 1 & -1\\ \textbf{0} & \textbf{0} & \textbf{0} \end{array} \right)
+ 14\cdot \left(\begin{array}{rrr}  0 & 1 & -1 \\ 0 & 0 & 0  \\  0 & 0 & 0\\ 0 & 0 & 0\\ 1 & -1 & 0 \\  -1 & 0 & 1\\\textbf{0} & \textbf{0} & \textbf{0}\end{array} \right) \\
+ 7\cdot &\left(\begin{array}{rrr}  0 & 0 & 0 \\ -1 & 0 & 1 \\  0 & 0 & 0\\ 0 & 0 & 0\\ 0 & 0 & 0\\ 1 & -1 & 0 \\ 0 & 1 & -1 \end{array} \right)
+ 7\cdot \left(\begin{array}{rrr}  0 & 1 & -1 \\ 0 & 0 & 0  \\  0 & 0 & 0\\ 0 & 0 & 0\\ 0 & 0 & 0\\ 1 & -1 & 0 \\  -1 & 0 & 1\end{array} \right)
=0
\end{align*}
}
with $\sum_i\bar h_i=57+24+40+16+24+48+56+28+28+14+14+7+7=363$.

However, if we read the relation for $M=6$ differently and choose $l=0$ and $s=-h_0=-25$, we obtain:

\smallskip

$M=7$
{\footnotesize
\begin{align*}
 25\cdot &\left(\begin{array}{rrr} 0 & -1 & 1 \\ 1 & 0 & -1 \\ 0 & 0 & 0 \\ 0 & 0 & 0\\ 0 & 0 & 0\\ \textbf{0} & \textbf{0} & \textbf{0}\\  -1 & 1 & 0 \end{array} \right)
+ 24\cdot \left(\begin{array}{rrr} 0 & 0 & 0 \\ 0 & 1 & -1\\-1 & 0 & 1 \\ 0 & 0 & 0\\ 0 & 0 & 0\\ 1 & -1 & 0\\ \textbf{0} & \textbf{0} & \textbf{0}\end{array} \right)
+ 40\cdot \left(\begin{array}{rrr} -1 & 0 & 1 \\ 0 & 0 & 0 \\ 0 & 1 & -1 \\ 0 & 0 & 0\\ 0 & 0 & 0\\ 1 & -1 & 0\\ \textbf{0} & \textbf{0} & \textbf{0} \end{array} \right)\\
+ 16\cdot &\left(\begin{array}{rrr} -1 & 1 & 0 \\ 0 & 0 & 0 \\ 0 & -1 & 1 \\ 1 & 0 & -1\\ 0 & 0 & 0 \\ 0 & 0 & 0\\ \textbf{0} & \textbf{0} & \textbf{0}\end{array} \right)
+ 24\cdot \left(\begin{array}{rrr} 0 & 0 & 0 \\ 0 & 1 & -1 \\ 1 & -1 & 0 \\ -1 & 0 & 1\\ 0 & 0 & 0 \\ 0 & 0 & 0\\ \textbf{0} & \textbf{0} & \textbf{0}\end{array} \right)
+ 48\cdot \left(\begin{array}{rrr} 0 & 1 & -1 \\ 1 & -1 & 0 \\ 0 & 0 & 0 \\ -1 & 0 & 1\\ 0 & 0 & 0 \\ 0 & 0 & 0 \\ \textbf{0} & \textbf{0} & \textbf{0}\end{array} \right)
+ 56\cdot \left(\begin{array}{rrr} 1 & -1 & 0 \\ -1 & 0 & 1 \\ 0 & 0 & 0 \\ 0 & 1 & -1\\ 0 & 0 & 0 \\ 0 & 0 & 0\\ \textbf{0} & \textbf{0} & \textbf{0}\end{array} \right) \\
+ 28\cdot &\left(\begin{array}{rrr}  0 & 0 & 0 \\ -1 & 0 & 1 \\  0 & 0 & 0\\ 1 & -1 & 0 \\ 0 & 1 & -1\\ 0 & 0 & 0\\ \textbf{0} & \textbf{0} & \textbf{0} \end{array} \right)
+ 28\cdot \left(\begin{array}{rrr}  0 & 1 & -1 \\ 0 & 0 & 0  \\  0 & 0 & 0\\ 1 & -1 & 0 \\  -1 & 0 & 1\\ 0 & 0 & 0\\ \textbf{0} & \textbf{0} & \textbf{0}\end{array} \right) 
+  14\cdot \left(\begin{array}{rrr}  0 & 0 & 0 \\ -1 & 0 & 1 \\  0 & 0 & 0\\ 0 & 0 & 0\\ 1 & -1 & 0 \\ 0 & 1 & -1 \\ \textbf{0} & \textbf{0} & \textbf{0}\end{array} \right)
+ 14\cdot \left(\begin{array}{rrr}  0 & 1 & -1 \\ 0 & 0 & 0  \\  0 & 0 & 0\\ 0 & 0 & 0\\ 1 & -1 & 0 \\  -1 & 0 & 1\\ \textbf{0} & \textbf{0} & \textbf{0}\end{array} \right) \\
+ 25\cdot &\left(\begin{array}{rrr}  0 & 0 & 0 \\ 1 & 0 & -1 \\  0 & 0 & 0\\ 0 & 0 & 0\\ 0 & 0 & 0\\ -1 & 1 & 0 \\ 0 & -1 & 1 \end{array} \right)
+ 25\cdot \left(\begin{array}{rrr}  0 & -1 & 1 \\ 0 & 0 & 0  \\  0 & 0 & 0\\ 0 & 0 & 0\\ 0 & 0 & 0\\ -1 & 1 & 0 \\  1 & 0 & -1\end{array} \right)
=0
\end{align*}
}
with $\sum_i\bar h_i=25+24+40+16+24+48+56+28+28+14+14+25+25=367>363$.

In fact, if one repeatedly applies Theorem \ref{Thm: New general bound} as in Corollary \ref{Cor: A3M bound} with $l=2$, $s=7$ up to some arbitrarily big $M_0$ and only then switches to different values $l=0$ and $s=-h_0$, we obtain the bound:
\[
	g(A_{3\times M})\geq\left(28-\frac{224}{2^{M_0}}\right)\cdot 2^{M-3}-\left(2^{M_0-1}-7\right), \text{ for } M\geq M_0.
\]
First note that a repeated application of Theorem \ref{Thm: New general bound} leads to a relation among elements in $\Graver(A_{3\times M_0})$ with $h_0=8\cdot 2^{M_0-4}-7$ and $\sum_i h_i=24\cdot 2^{M_0-3}-21$. Using this relation with $l=0$ and $s=-h_0<0$ (and $g=3$) in Corollary \ref{Cor: Recursion} we obtain
\begin{align*}
  g(A_{3\times M})&\geq 2^{M-M_0} \left(\sum_{i=0}^k h_i+\frac{-s}{1}\right)+\frac{s}{1}\\
  &=2^{M-M_0} ((24\cdot 2^{M_0-3}-21)+h_0)-h_0\\
  &=2^{M-M_0} ((24\cdot 2^{M_0-3}-21)+4\cdot 2^{M_0-3}-7)-2^{M_0-1}+7\\
  &=\left(28-\frac{224}{2^{M_0}}\right)\cdot 2^{M-3}-\left(2^{M_0-1}-7\right),
\end{align*}
for $M\geq M_0$, as claimed.

\bibliography{biblio}
\bibliographystyle{amsabbrvurl}

\end{document}